\documentclass[10pt]{amsart}
\let\Q\QQ
\newcommand{\Q}{\mathbb Q}
\newcommand{\Top}{\operatorname{Top}}

\input{corank1.sty}

\title[Enumerating rank $n$ bundles on $\mathbb CP^{n+1}$]{Enumerating complex rank $n$ vector bundles on $\mathbb CP^{n+1}$}

\author{Morgan P. Opie}
\address[Morgan P. Opie]{University of California, Los Angeles
}
\email{
mopie@math.ucla.edu}

\date{}                     

\begin{document}

\begin{abstract}
We enumerate complex rank $n$ topological vector bundles on $\mathbb CP^{n+1}$ with prescribed Chern classes. This extends work of Atiyah and Rees in the case $n=2$ and work of Hu in the case that all Chern classes are zero.
\end{abstract}
\maketitle
\tableofcontents

\section{Introduction}
Given a compact manifold $X$, a classical problem in topology is to enumerate complex topological vector bundles over $X$, up to isomorphism. This problem is stably approximated by the zeroth complex topological $K$-theory of $X$, which is accessible via well-established computational tools. The unstable problem is far less tractable. 

The simplest invariants of complex bundles are Chern classes, which are stable invariants valued in integral cohomology. Given fixed Chern classes, the number of rank $r$ bundles on $X$ with $i$-th Chern class equal to $a_i$ is finite. So we may ask:
\begin{q}\label{q:enumeration} Given an integer $r\geq 1$ and elements $a_i \in H^{2i}(X;\Z)$ for $1\leq i \leq r$, what is the number of isomorphism classes of complex rank $r$ topological bundles on $X$ with $i$-th Chern class equal to $a_i$?
\end{q}

In the case of complex projective spaces, \Cref{q:enumeration} has been extensively studied. For $r=1$, there is a unique line bundle on $\CP^m$ with first Chern class equal to an arbitrary element in $H^2(\CP^{m};\Z) \cong \Z$. For $r \geq m$, rank $r$ bundles on $\CP^m$ are determined by their integer-valued Chern classes. There is a divisibility condition, called the Schwarzenberger condition, $S_m$, such that $a_1,\ldots, a_m\in \Z$ are the Chern classes of a unique rank $r$ bundle on $\CP^m$ if and only if the integers satisfy $S_m$ (see \cite{Switzer, Thomas}, and \Cref{app:Schwarzenberger}).

For $1<r < m,$ there is no uniform answer to \Cref{q:enumeration}. However, small cases are known in full. The question for $m=3$ is resolved in seminal work of Atiyah and Rees \cite{AR}. The question for $m= 4$ is also known \cite{AR,Opie-r3p5,Switzer2,Yang23}. The case $m=5$ is answered by combining results of Switzer for rank $2$ bundles \cite{Switzer2}, results 
for rank $3$ bundles  \cite{Opie-r3p5}, and the $n=4$ case of \Cref{main:enumeration}(ii) below.  

Infinite families of results are also known. In \cite{Hu}, Hu addresses \Cref{q:enumeration} in the case that all Chern classes are zero using {\em Weiss calculus}. For $n\geq 2$, Hu enumerates bundles of rank $n$ on $\CP^{n+1}$ and of rank $n$ on $\CP^{n+2}$ with vanishing Chern classes \cite[Theorem 1.1 and Theorem 1.2]{Hu}. Our result generalizes Hu's first case to nonzero Chern classes. 

To state our main result, we introduce the following notation:
\begin{defn} Given $\vec a=(a_1,\ldots,a_n) \in\Z^n$, let $\phi(n, \vec a)$ denote the number of isomorphism classes of rank $n$ bundles on $\CP^{n+1}$ with $i$-th Chern class equal to $a_i$.
\end{defn}
\begin{thm}\label{main:enumeration} Let $n \geq 1$ and let $\vec a=(a_1,\ldots ,a_n) \in \Z^n$. Then:
\begin{enumerate}
\item[(i)] $\phi(n, \vec a)\neq 0$ if and only if $a_1,\ldots, a_n,0$ satisfy the Schwarzenberger condition $S_{n+1}$.
\item[(ii)] When the Schwarzenberger condition $S_{n+1}$ is satisfied, 
\[\phi(n,\vec a)=\begin{cases} 1 & \text{ if $n$ or $a_1$ is odd; }\\
2 &\text{ if $n$ and $a_1$ are even}.
\end{cases}  \]
\item[(iii)] Suppose that $n$ and $a_1$ are even and $a_1,\ldots, a_n,0$ satisfy the Schwarzenberger condition $S_{n+1}$. Then at most one of the two isomorphism classes of rank $n$ bundles on $\CP^{n+1}$ with $i$-th Chern class equal to $a_i$ extends to $\CP^{n+2}$. Exactly one extends if and only if the list of $n+2$ integers $a_1,\ldots, a_n,0,0$ satisfy the Schwarzenberger condition $S_{n+2}$. 
\end{enumerate}
\end{thm}
\begin{rmk} We refer the reader to \Cref{app:Schwarzenberger} for a discussion of the Schwarzenberger conditions.\end{rmk} 
Part (iii) above gives a qualitative invariant of rank $n$ bundles on $\CP^{n+1}$ in the case that the additional Schwarzenberger condition is satisfied: two non-isomorphic bundles with the same Chern classes are distinguished by whether or not they extend to $\CP^{n+2}$. However, we might hope for a computable, cohomological way to distinguish non-isomorphic bundles. This is a delicate problem: standard cohomological techniques often do not preserve enough information about the vector bundles with the same Chern classes. In some known cases, {\em twisted cohomological invariants} valued in generalized cohomology exist and (together with Chern classes) provide complete invariants of the topological isomorphism class of an unstable bundle \cite{AR,Opie-r3p5}. The author will address this secondary question in forthcoming work.

\subsection{Other related work}\label{history}

As a precursor to studying complex topological vector bundles on complex projective spaces, one might seek to understand vector bundles on spheres. Enumerating bundles on spheres amounts to understanding the unstable homotopy of unitary groups, which have been studied in a large range \cite{Toda59,MimuraToda_63,Matsunaga64,Oshima80}. In theory, this data plus information about the cell structure of complex projective spaces should enable analysis of fixed-rank bundles on $\CP^n$ for a large range of dimensions. However, this process is extremely difficult to carry out in a general way.

In \cite{CHO24}, the author, Chatham, and Hu 
work one prime $p$ at a time and use chromatic theories, {\em called higher real $K$-theories}, to obtain $p$-divisibility results for the number of rank $r$ bundles on $\CP^m$ with vanishing Chern classes for $m$ and $r$ in the {\em metastable range} and $m-r$ small relative to $p$. A key fact is that the enumeration of metastable, stably trivial rank $r$ bundles on $\CP^m$ reduces to calculating self-maps of stunted projective spectra \cite[Theorem 2.1]{Hu}. This problem exhibits periodic behavior and moreover there is a close relationship between the corank (dimension minus rank) and the chromatic height of the theories that can detect interesting bundles. When enumerating bundles with fixed nonzero Chern classes, the picture is not so clear. This work shows that, at least in small corank, the vanishing Chern class case may admit a direct extension.

\subsection{Outline}

In \Cref{Post}, we analyze the Postnikov tower for $BU(r)$ through dimension $2r+2$. The main input here is classical computations of homotopy groups of unitary groups. Using this analysis, we deduce that there are at most two rank $n$ bundles on $\CP^{n+1}$ with prescribed Chern classes. This approach also allows us to identify a geometrically defined action of $\Z/2$ on the set of rank $n$ bundles on $\CP^{n+1}$ with fixed Chern classes, which is used in \Cref{Extending}.
In \Cref{MoorePost}, we give an obstruction theory argument using the Moore--Postnikov tower for the map $BU(n) \to BU$. 
We analyze the obstruction-theoretic problem to prove part (ii) of \Cref{main:enumeration}.
In \Cref{Extending} we show that, when there are two non-isomorphic rank $n$ bundles on $\CP^{n+1}$ with the same Chern classes, at most one extends to $\CP^{n+2}$; we characterize exactly when one extends in terms of a Schwarzenberger condition. This completes the proof of \Cref{main:enumeration} (iii).
In \Cref{app:Schwarzenberger}, we discuss the Schwarzenberger condition and deduce (i) of \Cref{main:enumeration}.

\subsection{Notation and conventions}\label{sec:not}
Throughout this paper, we use the following notation and make the following conventions.
\begin{itemize}
\item Given spaces $X$ and $Y$, we write $[X,Y]$ for homotopy classes of maps from $X$ to $Y$. We write $Y^X$ for the space of maps from $X$ to $Y$, so that $[X,Y]=\pi_0(Y^X)$. 
\item Let $\Top$ denote the $\infty$-category of spaces.  While our main results can be stated in the homotopy category, or $1$-category, of topological spaces, the infinity category is needed for certain constructions to be well-defined. 
\item The $m$-truncation functor $\tau_{\leq m}\: \Top \to \Top_{\leq m}$ is the reflection onto the full subcategory $\Top_{\leq m}$ of spaces $X$ with $\pi_iX=0$ for $i>m$, i.e., the left adjoint to the inclusion $\Top_{\leq m} \hookrightarrow \Top.$
\item A space $Y$ is said to be $m$-skeletal if, for all $X \in \Top$, $[Y,X]\xrightarrow{\cong} [Y,\tau_{\leq m}X]$ via postcomposition with the $m$-truncation map $X \to \tau_{\leq m}X$. Examples of $m$-skeletal spaces include CW complexes with cells of dimension at most $m$ and manifolds of real dimension at most $m$.
\item Given a ring $R$ and a space or spectrum $X$, we write $H^*(X; R)$ for cohomology with coefficients in $R$. In the case that $R=\mathbb F_p$ for $p$ prime, we write $\HF_{p}^*(X)$ instead.
\item We write $\sphere$ for the sphere spectrum.
\item Given a space $X$, we write $\pi_iX$ for its unstable homotopy groups. For a spectrum $E$, we write $\pi_iE$ for its stable homotopy groups.
\item Given a group, ring, space, or spectrum $X$, we write $\c{X}{p}$ for its $p$-completion, where $p$ is a prime.
\item  Given integers $r<m$, the space $\CP^m_r$ is defined to be the cofiber of the map of spaces $ \CP^{r-1} \to \CP^{m}$. 
\end{itemize}

\subsection{Acknowledgements} I would like to thank a referee for \cite{Opie-r3p5} 
for asking about the enumeration of rank $4$ bundles on $\CP^5$, which led to this project. 
I am grateful to Mike Hill, Mike Hopkins, Yang Hu, and Alexander Smith for useful conversations.

This project was supported by a National Science Foundation Mathematical Sciences Postdoctoral Research Fellowship, Award No. 2202914.

\section{The Postnikov tower for $BU(n)$ through dimension $2n+2$}\label{Post}
The unstable homotopy groups of unitary groups $\pi_iBU(n)$ are known in a large range \cite{Mimura_HBAT}. In particular, for $i\leq2n+2$, we have:

\[ \pi_{i}BU(n) = \begin{cases} 0 & \text{    $i\leq 2n$, $i$ odd},\\
\Z &\text{    $i\leq 2n$, $i$ even},\\
\Z/n! &\text{    $i=2n+1$},\\
\Z/2&\text{    $i=2n+2$ \& $n$ even},\\
0&\text{    $i=2n+2$ \&  $n$ odd}.\\
\end{cases}
\]
Thus, for $n$ even, the Postnikov tower for $BU(n)$ is a tower of principal fibrations:

\begin{equation}\label[diagram]{post1}\begin{tikzcd}[column sep=.5em, row sep= 1em]
\tau_{\leq 2n+2}BU(n)\ar[d]\\
\tau_{\leq 2n+1} BU(n)\ar[d]\ar[r]&K(\Z/2,2n+3) \\
\tau_{\leq 2n} BU(n)\ar[d]\ar[r]&K(\Z/n!,2n+2) \\
\tau_{\leq 2n-2} BU(n)\ar[d]\ar[r]&K(\Z,2n+1) \\
\vdots\ar[d] \\
\tau_{\leq 4}BU(n) \ar[d]\ar[r]& K(\Z,7) \\
K(\Z,2) \ar[r] & K(\Z,5)  &.
\end{tikzcd}\end{equation}
For $n$ odd, the Postnikov tower has one fewer stage:
\begin{equation}\label[diagram]{post2}\begin{tikzcd}[column sep=.5em, row sep= 1em]
\tau_{\leq 2n+2}BU(n)\ar[d]\\
\tau_{\leq 2n} BU(n)\ar[d]\ar[r]&K(\Z/n!,2n+2) \\
\vdots\ar[d] \\
K(\Z,2) \ar[r] & K(\Z,5) &.
\end{tikzcd}\end{equation}
We refer the reader to \cite[Chapter IV]{GJ} or \cite[Chapter 3]{MP} for details on Postnikov towers and Moore--Postnikov towers (the relative version of the former).
\begin{prop}\label{prop:bound} Given $(a_1,\ldots, a_n)\in \mathbb Z^n$, there are at most two rank $n$ bundles on $\CP^{n+1}$ with $i$-th Chern class equal to $a_i$ if $n$ is even, and at most one if $n$ is odd.
\end{prop}
\begin{proof}
Since $\CP^{n+1}$ is $(2n+2)$-skeletal, $[\CP^{n+1},BU(n)]\simeq [\CP^{n+1}, \tau_{\leq 2n+2}BU(n)].$ We can build a map from $\CP^{n+1}$ to $\tau_{\leq 2n+2}BU(n)$ in stages using \Cref{post1} or \Cref{post2}. Let $i=2j$ for $j<n$. The obstruction to solving the lifting problem
\[\begin{tikzcd}
& &\tau_{\leq i+2}BU(n)\ar[d]\\
& \CP^{n+1} \ar[ur,dashed]\ar[r]&\tau_{\leq i}BU(n) \end{tikzcd}\]
is a class in $H^{i+1}(\CP^{n+1};\Z)=0$, since $i=2j$ is even. The choices of lift correspond to choices of $(j+1)$-st Chern class.

Let $\epsilon=1$ if $n$ is even and $\epsilon=2$ if $n$ is odd. The obstruction to solving the lifting problem
\[\begin{tikzcd}
& &\tau_{\leq 2n+\epsilon}BU(n)\ar[d]\\
& \CP^{n+1} \ar[ur,dashed]\ar[r]&\tau_{\leq 2n}BU(n) 
\end{tikzcd}\]
lies in $H^{2n+2}(\CP^{n+1};\Z/n!) \simeq \Z/n!$ and gives a condition on the Chern classes. The choices of lift are acted on transitively by a quotient of $H^{2n+1}(\CP^{n+1};\Z/n!)=0$. Referring to \Cref{post2}, this completes the argument for $n$ odd.
 
For $n$ even, we return to \Cref{post1}. The obstruction to lifting to $\tau_{\leq 2n+2} BU(n)$
lies in $H^{2n+3}(\CP^{n+1};\Z/2)=0$. The choices of lift are acted on transitively by a quotient of $H^{2n+2}(\CP^{n+1};\Z/2)\simeq \Z/2$. So, given $n$ even and integers $(a_1,\ldots, a_n) \in \mathbb Z^n$, there are at most $2$ rank $n$ bundles on $\CP^{n+1}$ with $i$-th Chern class equal to $a_i$.
\end{proof}

The above argument bounds $\phi(n,\vec a)$ using only the homotopy groups of unitary groups and cohomology of $\CP^{n+1}.$ To determine the number of isomorphism classes, we use a different method. Nevertheless, the proof of \Cref{prop:bound} shows how non-isomorphic bundles with the same Chern classes are related. 
\begin{prop}\label{prop:action} 
 If two rank $n$ bundles on $\CP^{n+1}$ have the same Chern classes, then one can be obtained from the other by an action of $\pi_{2n+2}BU(n)\simeq \Z/2$ on $[\CP^{n+1},BU(n)]$. This action is given by taking a pair 
$$\left(\sigma\: S^{2n+2} \to BU(n), \, V \: \CP^{n+1} \to BU(n)\right)$$ to the composite

\begin{equation}\label[diagram]{diag:squeeze}\CP^{n+1} \xrightarrow{s} \CP^{n+1} \vee S^{2n+2}  \xrightarrow{  V\vee \sigma} BU(n),\end{equation} where the first map $s$ is obtained by contracting the boundary of a $2n+2$ disc inside the top-dimensional cell of $\CP^{n+1}$.
\end{prop}
\begin{proof} This follows from transposing the Postnikov tower argument above across the skeleton-truncation adjunction on the homotopy category of spaces. Given a map $\CP^n \to BU(n)$, the obstruction to extending to a map $\CP^{n+1}\to BU(n)$ is a class in $\pi_{2n+1}BU(n)$. Choices of lift correspond to null-homotopies of this class, which are acted on transitively by $\pi_{2n+2}BU(n)$ via the action defined above.
\end{proof}

For more discussion on the action in \Cref{prop:action}, see \cite[Construction 1.6]{Opie-r3p5}. 

\section{The Moore--Postnikov tower for $BU(n) \to BU$}\label{MoorePost}

Fix an integer $n\geq 2$. Consider a map $c_{n+1}\: BU \to K(\Z,2n+2)$ representing the $(n+1)$-st Chern class. Let

\begin{equation}\label{eq:UF}F:= \op{fib}(c_{n+1}).\end{equation}

Consider the natural map $BU(n) \to BU$. The composite $BU(n) \to BU \xrightarrow{c_{n+1}} K(\Z,2n+2)$ is null, since the $(n+1)$-st Chern class of the universal bundle on $BU(n)$ is zero. Since $H^{2n+1}(BU(n);\Z)=0$, up to homotopy there is a unique lift  $f\: BU(n) \to F$ making the diagram below homotopy commutative
\begin{equation}\label[diagram]{diag:f}
\begin{tikzcd}
& F\ar[d] \\
BU(n) \ar[r]\ar[ur, dashed,"\exists ! f"]& BU \ar[r,"c_{n+1}"] & K(\Z,2n+2).
\end{tikzcd} 
\end{equation}
\begin{lemma}\label{lem:f_approx} The map $f\: BU(n) \to F$ is an equivalence after $(2n+1)$-truncation. 
\end{lemma}

We first study the cohomology of $F$. Consider the Serre spectral sequence for
\begin{equation}\label[diagram]{fibF}K(\Z,2n+1) \to F \to BU.\end{equation}  
\begin{defn}\label{graded_small}Given a graded module $M$ over the Steenrod algebra $\mathcal A$, and a non-negative integer $r$, let $M^{\leq r}$ denote the module over $\mathcal A$ obtained by taking the quotient of $M$ by the graded submodule consisting of terms of degree greater than or equal to $r+1$.\end{defn}
The main result we need is the following:
\begin{prop}[{\cite[Chapter 9, Theorem 3]{MT}}]\label{K9_cohomology} Let $p$ be a prime and let $\iota_{j}$ generate $\HF_p^{j}(K(\Z,j))$. Then
\[\HF_2^{*\leq 2n+3}(K(\Z,2n+1))\cong \mathbb F_2\{\iota_{2n+1},\Sq^2\iota_{2n+1}\},\] while for $p$ odd:
\[\HF_p^{*\leq 2n+3}(K(\Z,2n+1))\cong \mathbb F_p\{\iota_{2n+1}\}. \]
\end{prop}

We next note an elementary result, which can be proved with the Hurewicz theorem, the universal coefficient theorem, and the $E_2$-page of a Serre spectral sequence.
\begin{lemma}\label{lem:basic} Let $X$ and $Y$ be simply connected spaces with finitely generated integral cohomology. If $g\:X \to Y$ induces an isomorphism on cohomology in degrees less than or equal to $i$ with coefficients in $\mathbb Q$ and with coefficients $\mathbb F_p$ for all primes $p$, then $g$ induces an isomorphism on homotopy in degrees less than or equal to $i-1$, and a surjection in degree $i$.
\end{lemma}

\begin{proof}[Proof of \Cref{lem:f_approx}]
For the rational statement, recall that $$BU(n) \simeq_{\mathbb Q} \prod_{i=1}^n K(\mathbb Q,2i) \, \, \text{ and } \,\, BU\simeq_\Q \prod_{i=1}^\infty K(\Q,2i),$$ where $\simeq_{\mathbb Q}$ denotes an equivalence after rationalization. The rational equivalences are induced by the rational Chern class maps. Therefore we have a rational equivalence

$$F \simeq_{\Q} \prod_{\substack{i=1\\ i \neq n+1}}^\infty K(\Q,2i),$$ and the $2n+3$-truncation of $f$ induces a rational equivalence between $BU(n)$ and $F$, which is stronger than we need.

For any prime $p$, consider the $\HF_p$-Serre spectral sequence associated to the fiber sequence \Cref{fibF}, which has the form

$$E_2^{t,q}=\HF_p^t\left(BU, \HF_p^q(K(\Z,2n+1))\right)\implies \HF_p^{t+q}(F).$$
All classes in the cohomology of the fiber that contribute to the $E_\infty$-page in degree at most $2n+3$ are given in \Cref{K9_cohomology}. 
In particular, $\iota_{2n+1}$ is the first class in the cohomology of the fiber. 
By construction, $d_{2n+2}(\iota_{2n+1})=c_{n+1}.$ 
Since the natural map from $BU(n)\to BU$ induces an isomorphism on cohomology modulo the ideal generated by $c_i$'s for $i \geq n+1$, we see that $f$ induces an equivalence on cohomology through degree $2n+2$. We are done by \Cref{lem:basic}.
\end{proof}
The argument given above actually shows that, for $p$ an odd prime,
\[\c{\tau_{\leq2n+2}(F)}{p} \simeq \c{\tau_{\leq2n+2}BU(n)}{p},\]
since $\pi_{2n+2}BU(n)$ is either $\Z/2$ or $0$, depending on the parity of $n$. 
For $p=2$ the situation is more complicated.
We depict the $E_2$-page of the $p=2$ Serre spectral sequence for $F$ in \Cref{fig1}. The differentials on the $r$-page are $d_r^{t,q}\:E_r^{t,q} \to E_r^{t+r,q-r+1}.$ 

Consider the class $\Sq^2\iota_{2n+1}$ on the $E_2$-page of the spectral sequence. This class is transgressive by Kudo's transgression theorem and $d_{2n+4}(\Sq^2\iota_{2n+1})$ is the class of $\Sq^2c_{n+1}$. If $n$ is odd, $\Sq^2(c_{n+1})=c_1c_{n+1}+c_{n+2}$ and the class $\Sq^2\iota_{2n+1}$ does not survive the spectral sequence.
If $n$ is even, $\Sq^2(c_{n+1})=c_1c_{n+1}$ and the class $\Sq^2\iota_{2n+1}$ is a permanent cycle. So, in the case of $n$ even, there is a nonzero class in $\HF_2^{2n+3}(F)$ detected by $\Sq^2(\iota_{2n+1})$.

\begin{defn}\label{def:UG}For $n$ even, let $U\: F \to K(\Z/2,2n+3)$ be the cohomology class detected by $\Sq^2\iota_{2n+1}$ on the $E_2$-page of the $\HF_2$-Serre spectral sequence for the fibration of \Cref{fibF}.
Let $G$ denote the homotopy fiber of $U$.
\end{defn}
\noindent  Let $g\:BU(n)\to G$ be a lift of $f$ fitting into the diagram

\begin{equation}\label[diagram]{diag:g}
\begin{tikzcd}
& G\ar[d] \\
BU(n) \ar[r,"f"]\ar[ur, dashed," g"]& F\ar[d] \ar[r,"U"] & K(\Z/2,2n+3)\\
& BU\ar[r,"c_{n+1}"] &K(\Z,2n+2).
\end{tikzcd} 
\end{equation}

\begin{lemma}\label{lemG} Let $n$ be even. The map $\tau_{\leq 2n+2}g\: \tau_{\leq 2n+2} BU(n) \to \tau_{\leq 2n+2}G$ is an equivalence.
\end{lemma}
\begin{proof} Recall that $f $ is an equivalence on $k$-cohomology through degree $2n+3$ if $k=\mathbb Q$ or $k=\mathbb F_p$ with $p$ an odd prime. Since $K(\Z/2, 2n+3)$ is trivial after rationalization or $p$-completion at an odd prime, the same is true for $g$.

Note that $\HF_2^*(g)$ is an equivalence through degree $2n+2$. This follows from the Serre spectral sequence associated with the fibration $G \to F \to K(\Z/2,2n+3)$, where the cohomology of $G$ through degree $2n+2$ is identified with the submodule of the cohomology of $F$ generated by the preimages of $c_1,\ldots, c_n \in \HF_2^*(BU).$
By \Cref{lem:basic}, $\pi_*g$ is an isomorphism for $*\leq 2n+1$ and a surjection for $*=2n+2$. Since $\pi_{2n+2}BU(n) \simeq \pi_{2n+2}G \simeq \Z/2$, we are done.
\end{proof}

\begin{cor}\label{thm:Gworks} The map $g\: BU(n) \to G$ induces a bijection
\[ [\CP^{n+1}, BU(n)]\simeq [\CP^{n+1},G].\]
\end{cor}
$G$ is a particularly useful approximation to $BU(n)$ for our purposes. Consider \Cref{diag:g}. A map from $\CP^{n+1}$ to $BU$ is equivalent to a stable bundle; a lift to $G$ exists if and only if $c_{n+1}$ of the stable bundle is zero. 
This reproves a classical result:
\begin{rmk}\label{thm:suff} A stable complex bundle on $\CP^{n+1}$ has a rank $n$ representative if and only if its $(n+1)$-st Chern class vanishes.
\end{rmk}

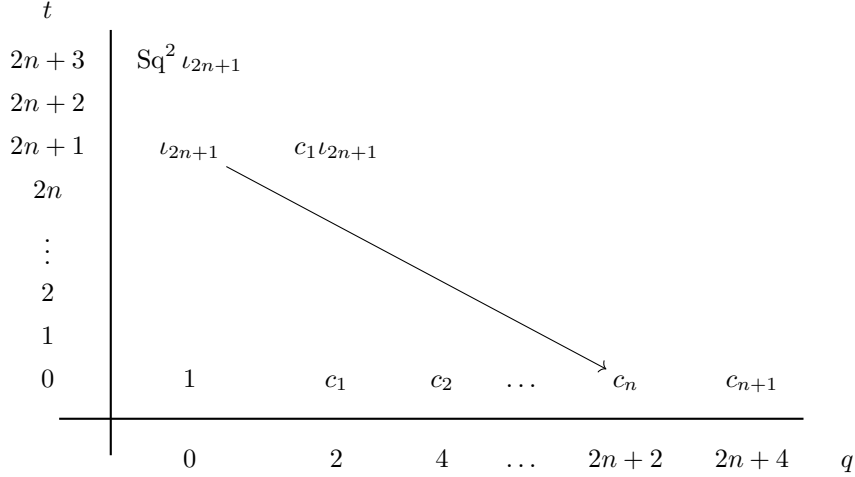
\begin{figure}[h]
\centering
\textbf{The $E_2$-page of the Serre spectral sequence computing $\HF_2^*(F)$.}\par\medskip
\begin{tikzpicture}
\matrix (m) [matrix of math nodes,
nodes in empty cells,nodes={minimum width=2.5ex,
minimum height=2.5ex,outer sep=-1pt},
column sep=.25ex,row sep=.25ex]{
t&\\
2n+3 & &\Sq^2\iota_{2n+1} & & & & & & & & & & && \\
2n+2&& & & & & & & & & & & && \\
2n+1 && \iota_{2n+1} & & c_1\iota_{2n+1} & & & & & & & & && \\
2n && & & & & & & & & & & && \\
\vdots\\
2 && & & & & & & & & & & && \\
1 && & & & & & & & & & & && \\
0&&1 & & c_1 & & c_2 && \hdots &&  c_{n+1} &  &c_{n+2} \\
&& & & & & & & & & & & && \\
\quad\strut & &0  & & 2  && 4 && \hdots& & 2n+2 & &2n+4& & q \strut \\};
\draw[->] (m-4-3.south east) -- (m-9-11.north west);
\draw[thick] (m-1-2) -- (m-11-2) ;
\draw[thick] (m-10-1) -- (m-10-14) ;\end{tikzpicture}\caption{Multiplicative generators for the $E_2$-page are indicated, as well as the class $c_1\iota_{2n+1}$ since it is significant.}\label{fig1}\end{figure}

\subsection{The number of representatives for a given stable bundle on $\CP^{n+1}$ with vanishing $(n+1)$-st Chern class} 
Let $n$ be even.
 By \Cref{prop:bound}, the number of homotopy classes of maps from $\CP^{n+1}$ to $G$ with fixed Chern classes is at most $2$. 

From the definition of $F$ and $G$, we have fiber sequences:
\begin{equation}\label{fib:G} G \to F \xrightarrow{U} K(\Z/2,2n+3)\end{equation}
and
\begin{equation}\label[diagram]{fib:F}K(\Z,2n+1) \xrightarrow{a} F \to BU.\end{equation}
From \Cref{diag:f}, a homotopy class $V\: \CP^{n+1} \to BU$ with 
vanishing $(n+1)$-st Chern class lifts uniquely to $F$. 
Given a lift $\tilde V\: \CP^{n+1} \to F$, the obstruction to lifting to 
$G$ lies in 
$H^{2n+3}(\CP^{n+1};\Z/2)=0$. Let 
\begin{equation}\label[diagram]{eq:ucirc}U\circ -\: F^{\CP^{n+1}} \to K(\Z/2,2n+3)^{\CP^{n+1}}\end{equation}
be the morphism on mapping spaces given by post-composition with $U$. We point $F^{\CP^{n+1}}$ by $\tilde V$ and we point $K(\Z/2,2n+3)^{{\CP^{n+1}}}$ by the zero map. Applying $\pi_1$ to \Cref{eq:ucirc}, we obtain a map
\begin{equation}\label{eq:pi1u} \pi_1(U\circ -) \: \pi_1(F^{\CP^{n+1}}) \to \pi_1(K(\Z/2,2n+3)^{{\CP^{n+1}}}). 
\end{equation} 
By \cite[Lemma 2.3]{Opie-r3p5}, we have the following:
\begin{lemma}\label{lem:Ucirc} Let $\pi_1(U\circ -)$ be as in \Cref{eq:pi1u}. The choices of lift of $\tilde V\: \CP^{n+1}\to F$ to $G$ are a torsor for
$\op{Coker}\left(\pi_1(U\circ-)\right)$. \end{lemma}

To compute $\op{Coker}\left(\pi_1(U\circ -)\right)$, we make some auxiliary observations.
\begin{lemma}\label{lem:surj} The map $a\:K(\Z,2n+1) \to F$ from \Cref{fib:F} induces a surjection
\[ \pi_1(a\circ -)\: \pi_1\left(K(\Z,2n+1)^{\CP^{n+1}} \right) \to \pi_1\left( F^{\CP^{n+1}}\right) .\]
\end{lemma}
\begin{proof} Consider the 
fiber sequence of mapping spaces 
\[K(\Z,2n+1)^{\CP^{n+1}} \xrightarrow{a\circ -} F^{\CP^{n+1}} \to BU^{\CP^{n+1}},\]
where $F^{\CP^{n+1}}$ and $BU^{\CP^{n+1}}$ are pointed by $\tilde V$ and $V$, respectively, and the fiber is pointed by zero. The induced long exact sequence on homotopy includes a portion
\[ \pi_1(K(\Z,2n+1)^{\CP^{n+1}}) \xrightarrow{\pi_1(a\circ -)} \pi_1(F^{\CP^{n+1}}) \to \pi_1(BU^{\CP^{n+1}})\simeq K^1(\CP^{n+1})=0.\]
\end{proof}
\Cref{lem:surj} immediately implies the following two results:
\begin{cor}\label{surj:cor1} Consider the action of 
$\pi_1(K(\Z,2n+1)^{\CP^{n+1}})$ on $\pi_1(F^{\CP^{n+1}})$ given on $s \in \pi_1(K(\Z,2n+1)^{\CP^{n+1}})$ and $x \in \pi_1(F^{\CP^{n+1}})$ by 
\[(s,x) \mapsto \left(\pi_1(a\circ -)(s)\right)\cdot x.\]
This action is transitive. 
\end{cor}
\begin{cor}\label{surj:cor2}
Let $\pi_{1}(U\circ -)$ be as in \Cref{eq:pi1u}. Let $ \tilde V\: \CP^{n+1} \to F$ be a given basepoint in the mapping space $F^{\CP^{n+1}}$, and let $A \in \pi_1( F^{\CP^{n+1}})$ be arbitrary. Then
\[ \op{Im}\left( \pi_1(U\circ - )\right) =
\left\{ \pi_1(U\circ -)(\sigma \cdot A) \, \,| \,\, \sigma \in \pi_1( K(\Z,2n+1)^{\CP^{n+1}})\right\}. \]
\end{cor}
\begin{defn}\label{defn:coaction}
Let $m\: K(\Z,2n+1) \times F \to F$ denote the natural action of the fiber on the total space in the fibration $K(\Z,2n+1) \to F \to BU.$ Let \[m^*\: \HZt^*(F) \to  \HZt^*\left(K(\Z,2n+1)\times F\right) \] be the induced coaction on cohomology. \end{defn}
Thus, $m^*U=U \circ m$ represents a class in $\HZt^{2n+3}\left(K(\Z,2n+1) \times F\right)$. Let $A \: S^1 \times \CP^{n+1} \to F$ represent an arbitrary class in $\pi_1(F^{\CP^{n+1}})$, meaning that \[A|_{\{0\} \times \CP^{n+1}} \simeq \tilde V\: \CP^{n+1} \to F.\]
Consider the (non-commutative) diagram:
\begin{equation}\label[diagram]{cd:action}
\begin{tikzcd}
{S^1 \times \CP^{n+1}}\ar[r,"A"] \ar[dr,"{(\iota_1t^n, A)}\,\,\,\,\,\,\,\,\,\,\,\," below]
& F \ar[r,"U"] 
& K(\Z/2,2n+3) \\
& {K(\Z,2n+1) \times F} \ar[u,"m"]\ar[ur,dashed,"{m^*U}"],
\end{tikzcd}
\end{equation}
where:
\begin{itemize}
\item $\iota_1$ represents a generator for $H^1(S^1;\Z)$; 
\item $t$ represents a generator for $H^2(\CP^{n+1};\Z)$; and therefore
\item $\iota_1t^n$ represents a generator for $H^{2n+1}(S^1 \times \CP^{n+1};\Z).$
\end{itemize}
Let $\sigma$ generate $\pi_1(K(\Z,2n+1)^{\CP^{n+1}}).$ Note that the cohomology class represented by $$\pi_1(U \circ -)(\sigma \cdot A) $$ is equal to $$m^*U(\iota_1t^n,A)\in \HZt^{2n+3}(S^1\times \CP^{n+1})\simeq \pi_1\left( K(\Z/2,2n+3)^{\CP^{n+1}}\right).$$ Thus, by \Cref{surj:cor2}:
\begin{cor}\label{cor:rephrase1} $\op{Coker}\left(\pi_1(U\circ -)\right)=0$ if and only if $(\iota_1t^n,A)^*m^*U \neq A^*U$.
\end{cor}

To apply \Cref{cor:rephrase1}, we investigate the cohomology classes $U$ and $m^*U$ in more detail.
\begin{prop}\label{prop:coaction} Consider $\HZt^*\left(K(\Z,2n+1)\times F\right) \simeq \HZt^*\left(K(\Z,2n+1)\right) \otimes  \HZt^*\left(F\right).$  Let $\iota_{2n+1}'$ generate $\HZt^{2n+1}\left(K(\Z,2n+1)\right)$.
Let $U \in \HZt^*(F)$ be as in \Cref{def:UG} and let $m$ and $m^*$ be as in \Cref{defn:coaction}. Then
\[m^*U=1\otimes U +\Sq^2\iota'_{2n+1}\otimes1 +\iota'_{2n+1}\otimes c_1,\]
where $c_1\in \HZt^2(F)$ is the pullback of the universal first Chern class in the cohomology of $BU$.
\end{prop}
\begin{proof}
Recall from \Cref{def:UG} that $U$ is the cohomology class detected on the 
$E_\infty$-page of the Serre spectral sequence for $\HZt^*(F)$ by the class of $\Sq^2\iota_{2n+1}$ (see \Cref{fig1}). 
To compute $m^*U$, we need to find the correct representative on the $E_2$-page. 
The relation $d_{2n+4}(\Sq^2\iota_{2n+1})=[c_1c_{n+1}]=d_{2n+2}(c_1\iota_{2n+1})$ shows that, on the $E_1$-page, $U$ is represented by a cochain for the cohomology class $\Sq^2\iota_{2n+1} \otimes 1 + \iota_{2n+1}\otimes c_1$. This is a cochain in the double complex computing $\HZt^*(F)$ from cochains on $K(\Z,2n+1)$ tensored with cochains on $BU$.

To compute $m^*U$, note that the action of $K(\Z,2n+1)$ on $F$ induces a map of fibrations
\begin{equation}\label[diagram]{map_fibs}\begin{tikzcd}
K(\Z,2n+1) \times K(\Z,2n+1) \ar[r]\ar[d,"m_1"]& K(\Z,2n+1) \times F \ar[d,"m"]\ar[r]&BU\ar[d,"="] \\
K(\Z,2n+1) \ar[r] & F \ar[r] & BU,
\end{tikzcd}\end{equation}
where $m_1$ is the $H$-space multiplication on $K(\Z,2n+1)$.

The map of fibrations \Cref{map_fibs} induces a map of cohomological Serre spectral sequences in the opposite direction. The key analysis is of the Serre spectral sequence associated to \[K(\Z,2n+1) \times K(\Z,2n+1) \to K(\Z,2n+1) \times F \to BU. \] The $E_2$-page in the case $n=4$ is given in \Cref{fig2}.
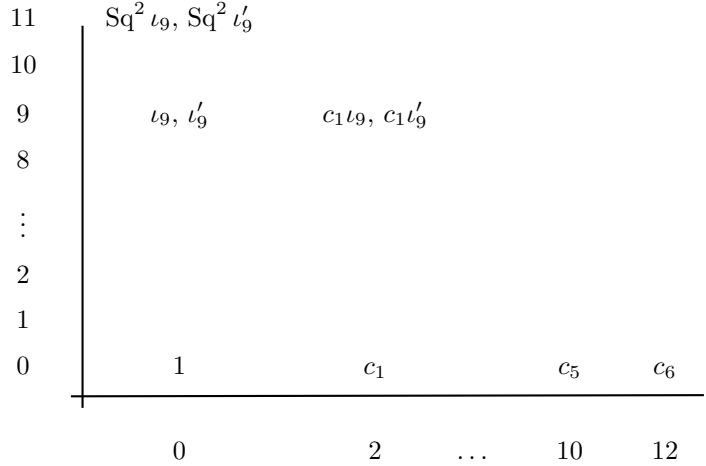
\begin{figure}[h]
\centering
\textbf{The $E_2$-page of the Serre spectral sequence computing $\HF_2^*\left(K(\Z,9) \times F \right)$ for $n=4$.}\par\medskip
\begin{tikzpicture}
\matrix (m) [matrix of math nodes,
nodes in empty cells,nodes={minimum width=2.5ex,
minimum height=2.5ex,outer sep=-1pt},
column sep=1ex,row sep=.5ex]{
11 & & \Sq^2\iota_9, \, \Sq^2 \iota_9'  & & & & & && \\
10 && &  & & & & && \\
9 && \iota_9, \, \iota_9' & & c_1\iota_9,\, c_1\iota_9'  && && && & \\
8 && && && && & \\
\vdots && && && && &\\
2 && &&  && && & \\
1&& && && && & \\
0&&1 && c_1 && &c_5 &  &c_6 \\
& & && && && & \\
\quad\strut 
&& 0 &  & 2&  \hdots  && 10 &&12 \strut \\};
\draw[thick] (m-1-2.east) -- (m-9-2.east) ;
\draw[thick] (m-9-2.north) -- (m-9-13.south) ;\end{tikzpicture}\caption{Multiplicative generators for the $E_2$-page are indicated, as well as the classes $c_1\iota_9$ and $c_1\iota_9'$.}\label{fig2}\end{figure}

Since $m_1^*(\iota_{2n+1})=\iota_{2n+1}+\iota_{2n+1}'$, we find that on the $E_2$-page:
\[m^*(\Sq^2\iota_{2n+1}+c_1\iota_{2n+1})=\Sq^2\iota_{2n+1}+c_1\iota_{2n+1}+\Sq^2\iota_{2n+1}'+c_1\iota_{2n+1}'.\]
Passing to $E_\infty$-pages, we obtain the result.
\end{proof}

\begin{cor}\label{cor:difference} Let $\tilde V\: \CP^{n+1} \to F$ represent a lift of $V\: \CP^{n+1} \to BU$. Suppose that $c_1(V)=a_1t$, where $ a_1\in\mathbb Z$. Let $A \in \pi_1(F^{\CP^{n+1}}, \tilde{V})$ be arbitrary. Then 
\[(\iota_1t^{n},A)^*m^*U-A^*U= a_1\iota_1t^{n+1} \in  \HZt^{2n+3}(S^1 \times \CP^{n+1}).\] 
\end{cor}
\begin{proof} By \Cref{prop:coaction}, $m^*U-U=\Sq^2\iota_{2n+1}'+c_1\iota_{2n+1}'.$

Therefore \begin{align*}(m^*U-U)(\iota_1t^{n},A)&=\Sq^2(\iota_1t^n)+c_1(V)\iota_1t^n\\
&= c_1(V)\iota_1 t^n,\end{align*}
since $\Sq^2(\iota_1t^n)=\iota_1\Sq^2(t^n)=0.$
\end{proof} 
We combine \Cref{cor:rephrase1}, \Cref{cor:difference},  and the fact that $\pi_1\left(K(\Z/2,2n+3)^{\CP^{n+1}}\right)=\Z/2$:
\begin{cor}\label{cor:rephrase} Let $n$ be even and let $V\: \CP^{n+1} \to BU$ be fixed and let $\tilde V\: \CP^{n+1} \to F$ lift $V$. Let $U$ be as in \Cref{def:UG}. Then
 \[\op{Coker}\left(\pi_1(U\circ -) \right)= \begin{cases}\Z/2& \text{ if  $c_1(V)\equiv 0 \pmod 2$ } \\
0 & \text{ if $c_1(V) \equiv 1 \pmod 2$}\end{cases}.\]
\end{cor}
Combining \Cref{cor:rephrase}, \Cref{lem:Ucirc}, and \Cref{prop:bound}, we prove part (ii) of \Cref{main:enumeration}:
\begin{thm}\label{main:thm1}
Let $n$ be even and let $V\: \CP^{n+1} \to BU$ represent a stable bundle with $c_{n+1}(V)=0$. If $c_1(V)$ is odd, there is a unique rank $n$ representative for $V$, up to isomorphism. If $c_1(V)$ is even, there are two non-isomorphic rank $n$ representatives for $V$.
\end{thm}

\section{Extending vector bundles on $\CP^{n+2}$}\label{Extending}

Our goal in this section is to prove part (iii) of \Cref{main:enumeration} (see \Cref{thm:iii} below). We begin with two preliminary lemmas.
\begin{lemma}\label{lem:gen}
Let $n$ be even and let $a\: S^{2n+3} \to S^{2n+2}$ denote an unstable representative for $\eta \in \pi_1\sphere$ (the first nonzero two-torsion class in the stable homotopy of spheres). Let $\sigma$ generate $\pi_{2n+2}BU(n) \simeq \Z/2$. The class $\sigma \circ a$ generates nonzero two-torsion in $ \pi_{2n+3}BU(n)$.
\end{lemma}
\begin{proof}
By \cite[Theorem 4.3]{Toda59}, for $n$ even there is an isomorphism 
\[[S^{2n+2},BU(n)]\simeq \pi_{2n+2}( \Sigma^{\infty+1} \CP^{n+2}_n),\] 
where $\CP^{n+2}_n$ is a stunted projective space (see \Cref{sec:not}). In the framework of Weiss calculus (see \cite[Section 2]{Hu}), this correspondence is given in two steps. First, there are  isomorphisms $[S^{k}, T_1BU(n)]\simeq [S^k,BU(n)]$ for $k<4n$, where $T_1BU(n)$ is the first stage in the Weiss tower. This identification is compatible with precomposition by maps between spheres in this range. Second, there is a diagram
\[
\begin{tikzcd}
L_1BU(n) \ar[r] & T_1BU(n)\\
S^k \ar[ur] \ar[u,dashed]
\end{tikzcd}
\]
where $L_1BU(n)$ is the first Weiss layer and where the dashed arrow exists and is unique for $k$ even, $k<4n$.
Lastly, there is a natural identification $\pi_{k}L_1(BU(n))\cong \pi_{k}(\Sigma^{\infty+1} \CP^{n+2}_n)$ for $k \leq 2n+3$. 

An Adams spectral sequence for the two-localization of $\pi_{2n+2}(\Sigma^{\infty+1} \CP^{n+2}_n)$ shows that the generator  $\sigma\:S^{2n+2} \to BU(n)$ corresponds to $\eta$ on the bottom cell in $\Sigma^{\infty+1} \CP^{n+2}_n$ since $n$ is even. If $ \sigma \circ a$ were zero in $\pi_{2n+3}BU(n)$, then $\eta^2$ on the bottom cell of $\CP^{n+2}_n$ would be zero. But, $\eta^2$ on the bottom cell generates $\pi_{2n+3}(\Sigma^{\infty+1} \CP^{n+2}_n).$ \end{proof}

\begin{lemma}\label{lem:2torsion} Let $n\geq 4$ be even and let $V\: \CP^{n+1} \to BU(n)$ be a rank $n$ vector bundle on $\CP^{n+1}$. Let $h\: S^{2n+3} \to \CP^{n+1}$ be the Hopf map. If $a_1,a_2,\ldots, a_n,0,0$ satisfy the Schwarzenberger condition $S_{n+2}$, then the class $V \circ h$ lies in a $2$-torsion subgroup of $\pi_{2n+3}BU(n)$.
\end{lemma}
\begin{proof} 
Note that $\pi_{2n+3}BU(n) \cong \Z/(n+1)!\oplus \Z/2$ by \cite[pp. 970-971]{Mimura_HBAT}. Consider the fiber sequence $S^{2n+1} \to BU(n) \to BU(n+1)$. The relevant portion of the long exact sequence on homotopy is:

\[  \pi_{2n+3} BU(n) \to \pi_{2n+3}BU(n+1) \to \pi_{2n+2} S^{2n+1} \to \pi_{2n+2}BU(n) \to \pi_{2n+2}BU(n+1)\cong\Z. \]
Again referring to \cite[pp. 970-971]{Mimura_HBAT}, we get an exact sequence:
\[ \Z/(n+1)! \oplus \Z/2 \to \Z/(n+1)! \to \Z/2 \to \Z/2 \to 0,\]
which forces the kernel of the map $\pi_{2n+3} BU(n) \to \pi_{2n+3}BU(n+1) $ to be $\Z/2$.

Now, note that the composite 
\begin{equation}\label[empty]{extension1} S^{2n+3} \xrightarrow{h} \CP^{n+1} \xrightarrow{V} BU(n) \to BU(n+1)\end{equation}
is the obstruction to extending the rank $n+1$ bundle $V \oplus \underline{\mathbb C}$ on $\CP^{n+1}$, obtained by summing $V$ with a trivial line bundle, over $\CP^{n+2}$. We claim that \Cref{extension1} is null. Given the claim, the map $V \circ h$ takes values in the kernel that we have already identified with $\Z/2$.

To prove the claim, note that, by \Cref{cor:Schwarz}, the $(n+2)$-tuple $(a_1,\ldots, a_n,0,0)$ satisfies $S_{n+2}$ if and only if there is a rank $(n+1)$ vector bundle $\tilde V$ on $\CP^{n+2}$ with $i$-th Chern class equal to $a_i$ for $1\leq i \leq n,$ and $(n+1)$-st Chern class equal to zero. Since rank $n+1$ bundles on $\CP^{n+1}$ are stable, $V \oplus \underline{\mathbb C}$ and $\tilde{V}|_{\CP^{n+1}}$ are isomorphic. In particular, $V \oplus \underline{\mathbb C}$ extends over $\CP^{n+2}$.
\end{proof} 

\begin{thm}\label{thm:iii} Suppose that $n\geq 2$ is even, that we are given integers $a_1,\ldots, a_n,0$ satisfying the Schwarzenberger condition $S_{n+1}$, and that $a_1$ is even. Of the two rank $n$ bundles on $\CP^{n+1}$ with $i$-th Chern class equal to $a_i$, at most one extends to a rank $n$ bundle on $\CP^{n+2}$. Moreover, exactly one extends if and only if $a_1,\ldots, a_n,0,0$ satisfies the Schwarzenberger condition $S_{n+2}$.\end{thm}

\begin{proof} Note that in the case $n=2$, this is proved in \cite[Section 7]{AR}. 

For arbitrary even $n$, we use the set-up from \Cref{lem:gen} and \Cref{prop:action}. The maps that will be relevant are:
\begin{itemize}
\item $h\: S^{2n+3} \to \CP^{n+1}$, the Hopf map; 
\item  $p\: \CP^{n+1} \to S^{2n+2}$, the quotient map onto the cofiber of the cellular inclusion $\CP^{n} \to \CP^{n+1}$; and
\item  $\sigma\: S^{2n+2} \to BU(n)$ that generates $\pi_{2n+2}BU(n)$.
\end{itemize}
We will also consider the action of $\pi_{2n+2}BU(n)$ on $[\CP^{n+1}, BU(n)]$, sending a pair $(\tau, V) \in \pi_{2n+2}(BU(n)) \times [\CP^{n+1},BU(n)]$ to the composite
$$\CP^{n+1} \xrightarrow{s} \CP^{n+1} \vee S^{2n+2} \xrightarrow{V \vee \tau} BU(n).$$

Given $V \: \CP^{n+1} \to BU(n)$, consider the composite $V \circ h \in \pi_{2n+3}BU(n)$. The Hopf map is the attaching map for the $(2n+4)$-cell in $\CP^{n+2}$, so an extension of $V$ to $\CP^{n+2}$ exists if and only if $V \circ h$ represents the trivial homotopy class. We claim that $V \circ h-(\sigma V)\circ h$ is nontrivial in $\pi_{2n+3}BU(n)$, where $\sigma$ generates $\pi_{2n+2}BU(n)$. 

Given the claim, we note the following. For $a_1$ even, the two non-isomorphic rank $n$ vector bundles on $\CP^{n+1}$ with Chern classes $a_1,\ldots, a_n$ are precisely $V$ and $\sigma V$. The claim shows the obstructions $V\circ h$ and $\sigma V\circ h$ are different, so at most one of $V$ and $\sigma V$ can extend. In the case that $(a_1,\ldots, a_n,0,0)$ satisfy $S_{n+2}$, both obstructions take values in a nonzero two-torsion group by \Cref{lem:2torsion}, so exactly one of the obstructions must vanish. Conversely, if $S_{n+2}$ is not satisfied, there is no stable bundle on $\CP^{n+2}$ with prescribed Chern classes (see \Cref{lem:Thomas_thmA}), and so no rank $n$ bundle on $\CP^{n+2}$ with the prescribed Chern classes. Thus, the claim proves the theorem.

For the claim, we follow \cite[Proof of Lemma 2]{Switzer2} and find:
\[(\sigma V) \circ h- V\circ h=(n+1)(\sigma \circ a) +[V \circ i_2,\sigma]\] where $i_2\: S^{2} \to \CP^{n+1}$ represents a generator of $\pi_2(\CP^{n+1})$ and the last term is a Whitehead product. Note that we use that $n$ is even to conclude that $a\simeq p \circ h.$ 

If $c_1(V)$ is even, then $V \circ i_2=2m$ for some $m\: S^2 \to BU(n)$. By properties of Whitehead products, $[V\circ i_2, \sigma] = 2 [m ,\sigma]=[m,2\sigma]$. Since $\sigma$ is $2$-torsion, it follows $[V\circ i_2,\sigma]=0$. In particular, 
\[ (\sigma V) \circ h - V \circ h=(n+1) (\sigma \circ a) = (\sigma \circ a),\]
which is a nonzero two-torsion class in $\pi_{2n+3}BU(n)$ by \Cref{lem:gen}.

\end{proof}

\appendix

\section{The Schwarzenberger conditions}\label{app:Schwarzenberger}

We follow \cite[Theorem 1]{Switzer} in this section. Given integers $c_1,\ldots, c_n\in \mathbb Z$, there exist complex numbers $\delta_1,\ldots, \delta_n \in \mathbb C$ such that 
$y^n+c_1y^{n-1}+\cdots + c_{n-1}y + c_n = \prod_{j=1}^n(y+\delta_j).$
 The Schwarzenberger condition is the requirement
\begin{equation}\label{schwarz} 
S_n\, \, : \forall r \in \Z \text{ such that }2 \leq r \leq n, \,\, \sum_{j=1}^n{\binom{\delta_j}{r}}  \in \mathbb Z \,.
\end{equation}
\begin{thm}[{\cite[Theorem A]{Thomas}, \cite[Theorem 1]{Switzer}}]\label{lem:Thomas_thmA}\label{thm:Schwarz} Integers $c_1,\ldots , c_k \in \Z$ are the Chern classes of a rank $k$ vector bundle on $\CP^k$ if and only if $c_1,\ldots, c_k$ satisfy $S_k$.
\end{thm}

\begin{cor}\label{cor:Schwarz} Integers $c_1,\ldots, c_n\in \Z$ are the Chern classes of a rank $n$ complex topological vector bundle on $\CP^{n+1}$ if and only if $c_1,\ldots, c_n, 0$ satisfy $S_{n+1}$.
\end{cor}
\begin{proof} Let $c_{n+1}=0$. By \Cref{thm:Schwarz}, $(c_1,\ldots, c_{n+1})$ are the Chern classes of a rank $n+1$ bundle on $\CP^{n+1}$ if and only if $S_{n+1}$ is satisfied. By \Cref{thm:suff}, a rank $n+1$ bundle on $\CP^{n+1}$ with Chern classes $(c_1,\ldots, c_{n+1})$ has a trivial rank $1$ summand if and only if $c_{n+1}=0.$ 
\end{proof}
The Schwarzenberger conditions are challenging to write explicitly in generality, but some formulas for rank $3$ bundles on $\CP^5$ are given in \cite[Section 2.4]{Opie-r3p5}. The computations there can be easily adapted to the case of rank $4$ bundles on $\CP^5$, or to rank $3$ bundles on $\CP^3$ or $\CP^4$.


\bibliographystyle{abbrv}
\bibliography{corank1}

\end{document}